\documentclass{article}
\usepackage{amssymb}
\usepackage{amsmath,color}
\usepackage[hidelinks]{hyperref}%%pacchetto per i links

\usepackage[square,numbers]{natbib}         %%%%round per parentesi tonde%%%
\usepackage{ntheorem}%%%%%Pacchetto per teoremi non numerati
\newtheorem{lem}{Lemma}[section]
\newtheorem{cor}[lem]{Corollary}
\newtheorem{teo}[lem]{Theorem}

\newtheorem{os}[lem]{Remark}

\newtheorem{prop}[lem]{Proposition}

\newenvironment{proof}{{\sc{Proof.}}}{\hfill\qed}

\newcommand{\qed}{\thinspace\null\nobreak\hfill\hbox{\vbox{\kern-.2pt\hrule
			height.2pt depth.2pt\kern-.2pt\kern-.2pt \hbox to2.5mm{\kern-.2pt\vrule
				width.4pt \kern-.2pt\raise2.5mm\vbox to.2pt{}\lower0pt\vtop
				to.2pt{}\hfil\kern-.2pt \vrule
				width.4pt \kern-.2pt}\kern-.2pt\kern-.2pt\hrule height.2pt depth.2pt
			\kern-.2pt}}\par\medbreak}

\newcommand{\R}{\mathbb{R}}

\newcommand{\C}{\mathbb{C}}

\newcommand{\eps}{\varepsilon}

\newcommand{\HS}{\R^{N+1}_+}

\newcommand{\ov}{\overline}

\def\Xint#1{\mathchoice
	{\XXint\displaystyle\textstyle{#1}}%
	{\XXint\textstyle\scriptstyle{#1}}%
	{\XXint\scriptstyle\scriptscriptstyle{#1}}%
	{\XXint\scriptscriptstyle\scriptscriptstyle{#1}}%
	\!\int}
\def\XXint#1#2#3{{\setbox0=\hbox{$#1{#2#3}{\int}$}
		\vcenter{\hbox{$#2#3$}}\kern-.5\wd0}}
\def\fint{\Xint -}

\newcommand{\ds}{\displaystyle}

\date{}
\begin{document}

	\title{
		Harnack inequality for  Bessel operators}
	\author{G. Metafune \thanks{Dipartimento di Matematica e Fisica ``Ennio De Giorgi'', Universit\`a del Salento, C.P.193, 73100, Lecce, Italy.
			e-mail:  giorgio.metafune@unisalento.it} \qquad L. Negro \thanks{Dipartimento di Matematica e Fisica ``Ennio De Giorgi'', Universit\`a del Salento, C.P.193, 73100, Lecce, Italy.
			e-mail:  luigi.negro@unisalento.it} \qquad C. Spina \thanks{Dipartimento di Matematica e Fisica  ``Ennio De Giorgi'', Universit\`a del Salento, C.P.193, 73100, Lecce, Italy.
			e-mail:  chiara.spina@unisalento.it}}
	
	\maketitle
	\begin{abstract}
		\noindent 
		We prove  uniqueness results and Harnack inequality for Bessel operators
		\begin{align*}
			%\label{def L transf alpha}
			D_t-\Delta_{x} -2a\cdot\nabla_xD_y- D_{yy}- \frac cy D_y 
			%	\nonumber	\\[1ex]&=y^{\alpha}\sum_{i,j=1}^{N+1}a_{ij}D_{ij}+y^{\alpha-1}\left(v,\nabla\right)-by^{\alpha-2}.
		\end{align*}
		in the strip $[0,T]\times \R^{N+1}_+=\{0 \leq t \leq T,  x \in \R^N, y>0\}$ under Neumann boundary conditions at $y=0$.
		
		\par

	\end{abstract}

	%%%%% Section 1 %%%%%
	\section{Introduction}
A classical result attributed to Tychonov  shows uniqueness of the heat equation $u_t=\Delta u$ under the pointwise growth condition $|u(t,x)| \leq Ce^{b|x|^2}$, $t>0,\ x \in \R^N$. More general results under integral growth assumptions, related to the decay of the heat kernel, have been proved later and lead to Widder's uniqueness theorem for positive solutions, see \cite{Widder}. A precise Harnack inequality for positive global solutions of the heat equation is then deduced by the explicit expression of the heat kernel.

In \cite{Friedman}, the author extends all these results to general parabolic operators with smooth, time dependent, coefficients. He uses the decay of the  fundamental solution of the operator and its adjoint and of some of their derivatives, to show uniqueness by integrating by parts.

In this paper we prove all these results for the singular parabolic operator $D_t-\mathcal L$ where 
\begin{align}\label{intro model operator}
		\mathcal L =\Delta_{x} +2a\cdot\nabla_xD_y+ D_{yy}+ \frac cy D_y 
	\end{align}
in the half-space $\R^{N+1}_+=\{(x,y): x \in \R^N, y>0\}$, under Neumann boundary conditions at $y=0$. We always assume  $c+1>0$  and  $a\in\R^N$ with $|a| <1$, which is equivalent to the ellipticity of the top order coefficients.

We denote by 
 $B_y$ the Bessel operator $D_{yy} +\frac{c}{y}D_y$ so that $\mathcal L=\Delta_x+2a\cdot\nabla_xD_y+B_y$ and $\mathcal L$ itself is named a Bessel operator. More general operators are easily obtained from $\mathcal L$ by linear change of variables, see \cite[Section 8]{MNS-Singular-Half-Space}.

Harnack inequality for Bessel operators has been studied in \cite{Garofalo1} in the case $a=0$ and $c \geq 0$. The author finds the very precise estimate for positive solutions
\begin{align*}
			0 <u(s,z_2)< \,  u(t,z_1) \left(\frac t s\right)^{\frac {N+1+c} 2}\exp\left({\frac{| z_1-z_2|^2}{4(t-s)}}\right),\quad  0<s<t,\quad  z_1,z_2\in\R^{N+1}_+
		\end{align*}
using specific properties of Bessel functions and the explicit knowledge of the heat kernel. When $ a \neq 0$, the heat kernel is not explicitly known and we use two-sided estimates which lead to a generic (but computable) constant $C$ instead of $\frac 14$. On the other hand, our methods work also for $-1<c<0$.
We prove, in all cases, 
\begin{align*}
			0<u(s,z_2)\leq \, C u(t,z_1) \left(\frac t s\right)^{\frac {N+1+c^+} 2}\exp\left({C\frac{| z_1-z_2|^2}{t-s}}\right),\quad  0<s<t,\quad  z_1,z_2\in\R^{N+1}_+.
		\end{align*}
In contrast with the results in \cite{Garofalo1}, we prove that our estimate holds for every positive solution of $D_tu-\mathcal L u=0$ and not only for the semigroup solution. This follows from our Widder-type result: all positive solutions are semigroup solutions.\\

We  now briefly outline our method and compare it with the existing literature.	
	
	The uniqueness problem for parabolic operators beyond the class studied in \cite{Friedman} attracted the interest of many authors.  We mention for example \cite{Pinchover} and \cite{Kogoj-Pinchover-Polidoro} where the authors considered respectively a class of parabolic operators with unbounded coefficients and some hypoelliptic  operators and proved uniqueness for the Cauchy problem after proving the validity of  uniform Harnack inequalities. Also in \cite{Ishige-Murata}  the authors study the uniqueness of nonnegative solutions of the Cauchy problem for parabolic equations in divergence form on manifolds or domains. In particular they prove  uniqueness under integral growth assumptions and uniqueness for nonnegative solutions via the parabolic Harnack inequality. Since we have sharp kernel estimates and parabolic regularity, our approach is the reverse of \cite{Ishige-Murata}. Indeed we first prove uniqueness under pointwise growth condition as in the classical case, by constructing a barrier,  then we convert integral bounds into pointwise bounds, by the parabolic regularity in \cite{dong2020neumann}. Finally,  we prove uniqueness results for nonnegative solution by using the lower bounds of the parabolic kernel. As a consequence, we get a  Harnack inequality for nonnegative solutions by using upper and lower bounds on the kernel.
	
In the last section we deduce Liouville type theorems for elliptic and parabolic equations. Liouville type theorems for the operators considered in this paper have already been studied in literature  in the case $a=0$, that is when mixed second order derivatives do not appear in the operator. See, for example, \cite[Theorem 1.2]{AudritoFioravantiVita2}, \cite[Theorem 1.6]{TerraciniTortoneVita1}.

%	\begin{equation}\label{def L}
%		\mathcal L=\mbox{Tr }\left(AD^2\right)+\frac{\left(v,\nabla\right)}y
%		=\sum_{i,j=1}^{N+1}a_{ij}D_{ij}+\frac{ d\cdot \nabla_x+cD_y }{y}.
%	\end{equation}
% Here   $A=\left(a_{ij}\right)\in \R^{N+1,N+1}$ is a symmetric and positive definite 
	%$(N+1)\times (N+1)$
%	matrix and   we suppose $\frac c\gamma +1>0$, where $\gamma=a_{N+1,N+1}$; the vector $v=(d,c)\in\R^{N+1}$ satisfies  $d=0$ if $c=0$ i.e. $v$ is  oblique with respect %to the boundary of $\R^{N+1}_+$.  We    endow $\mathcal L$ with   Neumann or  oblique derivative boundary conditions at $y=0$
%	$$\lim_{y\to 0} D_y u=0\quad \text{(if $v=0$)},\qquad\qquad  \lim_{y\to 0}y^{\frac c\gamma}\, v \cdot \nabla u=0\quad \text{(if $c\neq 0$)}.$$
%	Up to some suitable linear transformations, the latter operator is equivalent to  the model operator 

	\bigskip
	\noindent\textbf{Notation.} For $N \ge 0$, $\R^{N+1}_+=\{z=(x,y): x \in \R^N, y>0\}$. We write $\nabla u, D^2 u$ for the gradient and the Hessian matrix of a function $u$ with respect to all $x,y$ variables and $\nabla_x u, D_y u, D_{x_ix_j }u, D_{x_i y} u$ and so on, to distinguish the role of $x$ and $y$.  For $m \in \R$ we consider the measure $y^m dx dy $ in $\R^{N+1}_+$ and  we write $L^p_m(\R_+^{N+1})$, and often only $L^p_m$ when $\R^{N+1}_+$ is understood, for  $L^p(\R_+^{N+1}; y^m dx dy)$. We use $d\mu(z)$ for the measure $y^c\, dx\, dy$, $z=(x,y) \in \R_+^{N+1}$ .
	%$\C^+$ stands for $\{ \lambda \in \C: \Rp \lambda >0 \}$ and, for $|\theta| \leq \pi$, we denote by  $\Sigma_{\theta}$  the open sector $\{\lambda \in \C: \lambda \neq 0, \ |Arg (\lambda)| <\theta\}$.  
	%We denote by  $\alpha^+$ and $\alpha^-$ the positive and negative  part of a real  number, that is $\alpha^+=\max\{\alpha,0\}$,  $\alpha^-=-\min\{\alpha,0\}$. 

%	We write often $(x,y)$ or $x\cdot y$ to denote the inner product of $\R^N$ and, for $A,B\in\R^{N,N}$ symmetric, $\mbox{Tr }\left(AB\right)=\sum_{i,j}a_{ij}b_{i,j}$. %Moreover, if $\omega\in\R^N$, we write also  $\omega\otimes\omega\in\R^{N,N}$ to denote the matrix $\left(\omega_i\omega_j\right)_{i;j=1,\dots N}$; with this notation one has $\mbox{Tr }\left(\omega\otimes\omega  A\right)=\left(A\omega,\omega\right)$.
	Given $a$ and $b$ $\in\R$, $a\wedge b$, $a \vee b$  denote  their minimum and  maximum. We  write $f(x)\simeq g(x)$ for $x$ in a set $I$ and positive $f,g$, if for some $C_1,C_2>0$ 
	\begin{equation*}
		C_1\,g(x)\leq f(x)\leq C_2\, g(x),\quad x\in I.
	\end{equation*}
	%Sometimes we also write $C=C(\alpha)$ to emphasize, for a positive  constant $C>0$, its dependence on the parameter $\alpha$. If $d\nu$ is a measure we write $\fint_Q ud\nu=\frac{1}{\nu(Q)}\int_Qud\nu$ to denote the mean of a function $u$ over a set $Q$.

	\medskip
	\noindent\textbf{Acknowledgment.}
	The authors are members of the INDAM (``Istituto Nazionale di Alta Matematica'') research group GNAMPA (``Gruppo Nazionale per l’Analisi Matematica, la Probabilità e le loro Applicazioni'').

	\section{The  operator}\label{section general operator}
Let $c\in\R$, $a=(a_1, \dots, a_N) \in\R^N$ such that $c+1>0$, $|a|<1$.	We consider the singular elliptic operator
\begin{equation}  \label{La}
	\mathcal L :=\Delta_{x} +2\sum_{i=1}^Na_{i}D_{x_iy}+D_{yy}+\frac{c}{y}D_y=\Delta_x u+2a\cdot \nabla_xD_yu+ B_yu 
\end{equation} 
 endowed  with   the Neumann  boundary conditions 
	\begin{align}\label{Bound Cond}
		\lim_{y\to 0}y^{c}\, D_y u=0.
	\end{align}
	 The requirement $|a| <1$ is equivalent to the ellipticity of the top order coefficients.  	We recall the  results about generation of semigroups, maximal regularity and domain characterization  proved in \cite{MNS-Singular-Half-Space, Negro-AlphaDirichlet}. 
	We work in the spaces $L^p_m(\R_+^{N+1}):=L^p(\R_+^{N+1}; y^m dx dy)$, where $m\in\R$ and write only $L^p_m$ when $\R^{N+1}_+$ is understood and we recall that $d\mu(z)=y^c\, dx\, dy$.
	
	The Neumann  boundary condition \eqref{Bound Cond}   can be equivalently written in a integral form, 
	see  \cite[Proposition 4.6]{MNS-Sobolev}.  We define accordingly the  Sobolev spaces
	\begin{align*}
		W^{2,p}(m)=&\left\{u\in W^{2,p}_{loc}(\R^{N+1}_+):\    u, \nabla u, D^2u \in L^p_m\right\}
	\end{align*}
	and impose the boundary conditions by introducing
	\begin{align}\label{Definition W_N alpha}
		W^{2,p}_{\mathcal N}(m)&{=}\{u \in W^{2,p}(m):\ y^{-1}D_yu \in L^p_m\}.
	\end{align}

	\begin{teo}\label{Teorema generazione}{{\em(\cite[Theorem 7.7]{MNS-Singular-Half-Space})}}\label{Teo gen in Lpm}
		Let $0<\frac{m+1}p<c+1$. 		Then the operator
		\begin{align*}
			\mathcal L&=\Delta_x u+2a\cdot \nabla_xD_yu+ B_yu 
		\end{align*}
		endowed with domain 	$D(\mathcal L)=W^{2,p}_{\mathcal N}(m)$	generates a bounded analytic semigroup $(e^{t\mathcal L})_{t \geq 0}$  in $L^p_m$.
		% Moreover 
		%	\begin{align*}
			%		D(\mathcal L)=W^{2,p}_{v}(m).
			%	\end{align*} 
		%and the set  $ \mathcal C_v$ defined in \eqref{defC oblique} is a core for $\mathcal L$.
	\end{teo}
	%\begin{cor}\label{Oblique cor1}{{\em(\cite[Corollary 5.3]{Negro-AlphaDirichlet})}}
	%	Under the assumptions of the previous theorem, the estimate
	%	\begin{align}\label{elliptic regularity oblique} 
		%		\| D^2 u\|_{L^p_{m}} +\|y^{-1} v\cdot \nabla u\|_{L^p_{m}}\leq C\| \mathcal Lu\|_{L^p_{m}}
		%	\end{align}
	%	\mbox{}\\[-1.5ex]	holds for every $u \in W^{2,p}_{v}(m)$ (if $c=0$ replace $y^{-1} v\cdot \nabla u$ with $y^{-1} D_yu$).
	%\end{cor}
$\mathcal L$ has maximal regularity, that is if $u_t-\mathcal Lu=f$ with $u(0)=0$ and $f \in L^p$, then $u, \mathcal Lu \in L^p$.
	%\subsection{Sharp kernel estimates}
	
\medskip

	In \cite{Negro-Spina-SingularKernel, Negro-Spina-SingularKernel-Lower} we proved  that, if  $c +1>0$,  then the  heat kernel $p_{{\mathcal L}}$ of $\mathcal L$, written  with respect the measure $y^\frac{c}{\gamma}dz$, satisfies sharp estimates.  We recall here the main results.

	\begin{teo}{{\em(\cite[Theorems 4.18]{Negro-Spina-SingularKernel}, \cite[Theorem 7.16]{Negro-Spina-SingularKernel-Lower})}}\label{true.kernel} 
		Let $c+1>0$.  The semigroup   consists of integral operators i.e. there exists $p_{{\mathcal L}}(t,\cdot,\cdot)\in L^\infty(\R^{N+1}_+\times\R^{N+1}_+)$ 
		such that  for $t> 0$, $z_1=(x_1,y_1),\ z_2=(x_2,y_2)\in\R^{N+1}_+$
		\begin{align*}
			e^{t{\mathcal  L}}f(z_1)=
			\int_{\R^{N+1}_+}p_{{\mathcal L}}(t,z_1,z_2)f(z_2)\,d\mu(z_2),\quad f\in L^2_{{c}}.
		\end{align*}
		Moreover $p_{{\mathcal  L}}$ satisfies
		\begin{align}\label{up low section 1}
			p_{{\mathcal  L}}(t,z_1,z_2)
			\simeq C t^{-\frac{N+1}{2}} y_1^{-\frac{c}{2}} \left(1\wedge \frac {y_1}{\sqrt t}\right)^{\frac{c}{2}} y_2^{-\frac{c}{2}} \left(1\wedge \frac{y_2}{\sqrt t}\right)^{\frac{c}{2}}\,\exp\left(-\dfrac{|z_1-z_2|^2}{kt}\right),
		\end{align}
		where  $C,k$ are some positive constants which may differ in the upper and lower bounds. 
		%The upper estimate extends for complex $t$ on any subsector of $\Sigma_{\frac{\pi}{2}-\theta}$.
		%\begin{align*}
		%	|p_{{\mathcal  L}}(t,z_1,z_2)|
		%	\leq C_\epsilon |t|^{-\frac{N+1}{2}} y_1^{-\frac{c}{2\gamma}} \left(1\wedge \frac {y_1}{\sqrt {|t|}}\right)^{\frac{c}{2\gamma}} y_2^{-\frac{c}{2\gamma}} \left(1\wedge \frac{y_2}{\sqrt {|t|}}\right)^{\frac{c}{2\gamma}}\,\exp\left(-\dfrac{|z_1-z_2|^2}{k_\epsilon {|t|}}\right).
		%\end{align*}
	\end{teo}

	\begin{os}\label{oss equiv estimate}
		We emphasize that the above estimate can be written  equivalently  as
		\begin{align}\label{upper estimates ver2}
			p_{{\mathcal  L}}(t,z_1,z_2)
			\simeq C t^{-\frac{N+1}{2}}  y_i^{-c} \left(1\wedge \frac{y_i}{\sqrt t}\right)^{c}\,\exp\left(-\dfrac{|z_1-z_2|^2}{kt}\right),\qquad i=1,2
		\end{align}
		for some possibly different constants $C,k>0$. This is a consequence of \cite[Lemma 10.2]{MNS-Caffarelli} which says that for any $\epsilon>0$ there are $C_1,C_2>0$ such that for every $ y_1,y_2>0$
		\begin{align}\label{equiv estimates up low}
			C_1\,\exp\left(-\epsilon|y_1-y_2|^2\right)\leq \frac{y_1^{-\frac{c}2} \left(1\wedge y_1\right)^{\frac{c}2}}{y_2^{-\frac{c}2} \left(1\wedge y_2\right)^{\frac{c}2}}\leq C_2\,\exp\left(\epsilon|y_1-y_2|^2\right).
		\end{align}
	\end{os}
	\medskip

	\subsection{Remarks on the measure $d\mu(z)=y^c\,dx\,dy$}

	We work in $\R^{N+1}_+$ endowed with the standard euclidean metric  and, for $z_0=(x_0,y_0)\in\overline{\R^{N+1}_+}$ and $r>0$, we consider the cylindric balls 
	\begin{align}\label{def balls}
		Q(z_0,r):=\Big(B(x_0,r)\times ]y_0-r,y_0,y_0+r[\Big)\cap \HS=B(x_0,r)\times ](y_0-r)^+,y_0,y_0+r[.
	\end{align} 
	
	For $c+1>0$ we consider the measure $d\mu=y^c\,dx\,dy$ and we write 
	\begin{align*}
		V (z_0,r):=\mu\left(Q(z_0,r)\right)=\int_{Q(z_0,r)} y^c dz=w_Nr^{N}\int_{(y_0-r)^+}^{y_0+r} y^c dy.
	\end{align*}

	We need the following  elementary lemma, see \cite[Lemma 5.2]{MNS-PerturbedBessel}, which implies that $\mu$ is doubling. 
	
	\begin{lem}\label{Misura palle}
		Let $c+1>0$. 
		Then one has 
		\begin{align*}
			V(z_0,r)&=r^{N+1+c}\,V\left(\frac{z_0}r,1\right),\qquad 
			V(z_0,r)\simeq r^{N+1+c}\left(\frac{y_0}{r}\right)^{c}\left(\frac{y_0}{r}\wedge 1\right)^{-c}.
		\end{align*}
		In particular the function $V$ satisfies, for some  constants $C\geq 1$, the doubling condition
		%\begin{align*}
		%\frac{Q_c(y_0,s)}{Q_c(y_0,r)}\leq C \left(\frac{s}{r}\right)^{1\vee (c+1)},\qquad \forall y_0>0,\quad 0<r< s.
		%\end{align*}
		%In particular one has 
		%		\begin{align*}
			%			\frac{V(z_0,s)}{V(z_0,r)}\leq C \left(\frac{s}{r}\right)^N\left(1 \vee \frac{s}r\right)^{1+c^+},\qquad \forall\, s,  r>0.
			%		\end{align*}
			\begin{align*}
			\frac{V(z_0,s)}{V(z_0,r)}\leq C \left(\frac{s}{r}\right)^{N+1+c^+},\qquad \forall\, 0<r\leq s.
		\end{align*}
	\end{lem}

	%{\sc{Proof.}} Let $r>0$. A scaling argument immediately yields $V(z_0,r)=r^{N+M+\lambda}V\left(\frac{z_0}r,1\right)$. Without any loss of generality we can therefore assume  $r=1$. The local integrability of $|y|^\lambda$ implies that  $V_\lambda (z_0,1)$ is continuous as a function of $z_0$ and moreover $V_\lambda (z_0,1)\to \omega_N\int_{B(0,1)} |y|^{\lambda} dy>0$ as $y_0\to 0$; therefore  we can choose $R>1$ such that if  $|y_0|<R$ then 
	%\begin{align*}
	%V_\lambda (z_0,1)\simeq 1.
	%\end{align*}
	%On the other hand if $|y_0|>R$ then $|y|\simeq |y_0|$ for any $y\in B(y_0,1)$ which implies
	%\begin{align*}
	%V_\lambda (z_0,1)=w_N\int_{B(y_0,1)} |y|^{\lambda} dy\simeq |y_0|^{\lambda}.
	%\end{align*}
	%The last two inequalities yields $V_\lambda (z_0,1)\simeq \left(|y_0|\right)^{\lambda}\left(|y_0|\wedge 1\right)^{-\lambda}$. \qed

	%	\begin{os} \label{doubling}
		%		The previous inequality  implies that, for some positive constant $C>0$, one has
		%		%The weighted measure satisfies the doubling condition
		%		$$V(z_0,2r)\leq C2^{N+1+c^+} V(z_0,r),\qquad \forall z_0\in\ R^{N+1}_+,\ r>0.$$
		%		%	Moreover, recalling \eqref{equiv estimates up low}, one has for any $\epsilon >0$
		%		%	\begin{align*}
			%			%		\frac{1}{V(z_1,r)}\simeq 	\frac{1}{V(z_2,r)} \,\exp\left(\epsilon\frac{|y_1-y_2|^2}{t}\right),\qquad \forall t>0, z_1,z_2\in\R^{N+1}_+.
			%			%	\end{align*}
		%	\end{os}
The next result  provides a local estimate for $V$ which is valid for  $r$ and $y$ in bounded intervals. This is crucial for the proof of Proposition \ref{domination}.
	
	\begin{lem}\label{stima V locale}
		Let $c+1>0$ and let us fix $r_0,R_0>0$. Then, for some $C_1,C_2>0$ one has 
		\begin{align*}
			C_2\, r^{N+1+c^+}\leq 	V(z,r)\leq C_1\, r^{N+1-c^-}\qquad \forall\,r\in [0,r_0],\; y\in [0,R_0].
		\end{align*}
	\end{lem} 
	\begin{proof}
		The proof follows by an elementary calculation since by Lemma \ref{Misura palle} one has 
		\begin{align*}
			V(z,r)\simeq f(z,r):=r^{N+1+c}\left(\frac{y}{r}\right)^{c}\left(\frac{y}{r}\wedge 1\right)^{-c}=\begin{cases}
				r^{N+1+c},\quad & y\leq r\\[1ex]
				r^{N+1}y^c\quad & y>r.
			\end{cases}
		\end{align*} 
		Indeed if $c\geq 0$ then one has, for any $r\in [0,r_0]$, $y\in [0,R_0]$,
		\begin{align*}
			&	f(z,r)=
			r^{N+1+c}\leq r^{N+1}\,r_0^c\qquad\quad\; \text{if}\quad y\leq r\\[1ex]
			r^{N+1+c}\leq &f(z,r)=	r^{N+1}y^c\leq r^{N+1}\,R_0^c\qquad\quad  \text{if}\quad y> r.
		\end{align*}
		Similarly, if $c<0$,  
		\begin{align*}
			&	 r^{N+1}\,r_0^c\leq f(z,r)=
			r^{N+1+c}\qquad\quad\; \text{if}\quad y\leq r \\[1ex]
			r^{N+1}R_0^c\leq &f(z,r)=	r^{N+1}y^c\leq r^{N+1+c}\qquad\quad  \text{if}\quad y> r.
		\end{align*}
		These  inequalities prove the claim.
	\end{proof}
	We now rewrite the   bounds of Theorem \ref{true.kernel} in terms of  the measure of the cylindric balls. 
	
	\begin{os}\label{equiv balls}
		We remark that in  \cite{Negro-Spina-SingularKernel-Lower} the cylindric balls are defined  using right neighbourhoods in the $y$-variable, namely  
		$B(x_0,r)\times [y_0,y_0+r)$.  We adopt the equivalent  Definition \eqref{def balls}    in order to simplify some arguments. We  point out that by elementary calculation one has 
		 $$\mu\left(B(x_0,r)\times [y_0,y_0+r)\right)\simeq \mu\left(Q(z_0,r)\right)=V(z_0,r).$$ 
	\end{os}
	\begin{prop}\label{kernel tilde} Let us assume $c+1>0$. Then the heat kernel $p_\mathcal L$ of $\mathcal L$, written with respect to the measure $y^c\ dx\,dy$, satisfies  
		\begin{align}\label{up kernel measure}
			p_{\mathcal L}(t,z_1,z_2)
			&\simeq \frac{C}{V\left(z_1,\sqrt t\right)^{\frac 1 2}V\left(z_2,\sqrt t\right)^{\frac 1 2}}\exp\left(-\frac{|z_1-z_2|^2}{\kappa t}\right),\quad \forall \ t>0, \ z_1,\ z_2\in\R^{N+1}_+.
		\end{align}
		The above  estimates are equivalent, up to a change of  the constants $C,k>0$,  to
		\begin{align*}
			p_{{\mathcal  L}}(t,z_1,z_2)
			\simeq \frac{C}{V\left(z_i,\sqrt t\right)}\,\exp\left(-\dfrac{|z_1-z_2|^2}{kt}\right),\qquad i=1,2.
		\end{align*}	
	\end{prop}
	\begin{proof}
		These estimates follow from  Remark \ref{equiv balls},  Theorem \ref{true.kernel}, Lemma \ref{Misura palle} and \eqref{equiv estimates up low}.
	\end{proof}
	
	\section{Uniqueness under pointwise growth conditions}\label{Section guniqueness}
	
	We   
	%$$
%	\mathcal L =\Delta_{x} +2a\cdot \nabla_xD_yu+D_{yy}+\frac{c}{y}D_y, \quad |a|>1, c+1>0
%	$$ 
prove theorems of Tychonov type for solutions of the  parabolic problem
	\begin{align}\label{CP L}
		\begin{cases}
			\partial_t u=\mathcal Lu & 0<t \leq T,\, z\in\HS,\\[1ex]
			u(0,z) = u_0(z) & z\in\R^{N+1}_+,\\[1ex]
			\ds\lim_{y\to 0}y^{c}\, D_y u(t,x,y)=0&0<t \leq T, \, x\in\R^{N}.
		\end{cases}
	\end{align}
	% This will allow to deduce the uniqueness of the solution in some suitable functions classes and, in particular, its representation trough the integral kernel above recalled.
	We consider solutions of the Cauchy problem \eqref{CP L} satisfying
	\begin{align*}
		u\in C^{1,2}\left(]0,T]\times\ov{\R^{N+1}_+}\right)\cap C\left([0,T]\times\ov{\R^{N+1}_+}\right).
	\end{align*}
	This regularity assumption is not restrictive since by \cite[Section 4.1]{dong2020neumann} any weak solution of the equation $\partial_t u=\mathcal Lu$ satisfies $u\in C^{1,2}\left(]0,T]\times\ov{\R^{N+1}_+}\right)$. The continuity at $t=0$ is assumed here only   for simplicity  since all the results of these sections remain valid by regarding the  initial condition  $u(0,z) = u_0\in L^2_{loc}\left(\HS\right)$ in the weak sense $u(t,\cdot)\to u_0$ in $H^{-1}_{loc}\left(\HS\right)$.
	% :
	% \begin{align*}
		% 	\lim_{t\to 0}\int_{\HS}u(t,z)\varphi(z)\,y^cdz=\int_{\HS}u_0(z)\varphi(z)\,y^cdz,\qquad \forall \varphi\in C_c^1\left(\overline{\HS}\right)
		% \end{align*} 
	
	\smallskip 
	
	We start by proving an elementary result which clarifies the relationship between the weighted Neumann boundary condition $	\ds\lim_{y\to 0}y^{c}\, D_y u(x,y)=0$,  the assumption $\mathcal Lu\in C\left(\ov{\R^{N+1}_+}\right)$ and the usual Neumann boundary condition $D_yu(x,0)=0$.
	\begin{lem}\label{Boundary 0}
		Let $u\in C^2\left(\ov{\HS}\right)$. Then 
		$\mathcal Lu \in C\left(\ov{\HS}\right)$ if and only if  $ D_yu(\cdot,0)=0$.
		Moreover anyone of the above properties implies\quad  $	\ds\lim_{y\to 0}y^{c}\, D_y u\left(\cdot ,y\right)=0.$
	\end{lem}
	\begin{proof}
		Using Taylor's expansion we write  $$D_yu(x,y)=D_yu(x,0)+D_{yy}u(x,0)y+o(y).$$ The requirement $\mathcal Lu \in C\left(\ov{\HS}\right)$ is  equivalent to the existence of  $\ds\lim_{y\to 0} \frac{D_y}{y} u\left(x ,y\right)=\ell \in \C$ which is then equivalent to $D_yu(\cdot,0)=0$. In particular $\ell=D_{yy}(x,0)$.

Finally $\ds\lim_{y\to 0}y^{c}\, D_y u\left(\cdot ,y\right)=0$, by Taylor's expansion again, since $D_y(x,0)=0$.
\end{proof}
	The following lemma shows that if $u$ is a solution of \eqref{CP L}, the Neumann boundary condition  follows from the regularity.
	\begin{lem}  \label{boundary}
		Let  $u\in C^{1,2}\left(]0,T]\times\ov{\R^{N+1}_+}\right)$ be such that $\partial_t u-\mathcal Lu= 0$ in $]0,T]\times\ov{\R^{N+1}_+}$. Then $\ds D_y u(t,x,0)=0$ for $t\in ]0,T],\ x\in\R^{N}$.
	\end{lem}
	\begin{proof}
		The proof follows by Lemma \ref{Boundary 0} since the hypothesis $u\in C^{1,2}\left(]0,T]\times\ov{\R^{N+1}_+}\right)$ and the equation $\partial_t u-\mathcal Lu= 0$ imply that $\mathcal Lu\in C\left(]0,T]\times\ov{\R^{N+1}_+}\right)$.

	\end{proof}
	%\Red{***** Ho sostituito il Lemma \ref{boundary vecchio} sotto con la il Lemma \ref{Boundary 0} e la Proposizione \ref{boundary}. In particolare nel seguito possiamo non scrivere la condizione al bordo.}
	%\Red{
		%
		%		\begin{lem}  	*******Vecchia versione******\label{boundary vecchio}
			%		Let  $u\in C^{1,2}(]0,T]\times\ov{\R^{N+1}_+})\cap C([0,T]\times\ov{\R^{N+1}_+})$ be such that $\partial_t u-\mathcal Lu= 0$ in $]0,T]\times\ov{\R^{N+1}_+}$. Then $\lim_{y\to 0}D_y u(t,x,y)=0$ for $t\in ]0,T],\ x\in\R^{N}$.
			%	\end{lem}
		%	\begin{proof}
			%		Let $R, \delta>0$. By assumption 
			%		$$|\partial_t u(t,z)|=|\mathcal Lu(t,z)|\leq K$$ for some positive constant $K$ in $[\delta,T]\times\ov{B_R}\times[0,1]$. On the other hand 
			%		$|\Delta_{x} u(t,z)+2a\cdot \nabla_xD_yu+D_{yy}u(t,z)+D_{yy} u(t,z)|\leq C$ for some positive constant $C$ in $[\delta,T]\times\ov{B_R}\times[0,1]$. It follows that $\left|\frac{c}{y}D_y u(t,z)\right|\leq K$ in $]0,T]\times\ov{B_R}\times[0,1]$ and so  $\left|D_y u(t,z)\right|\leq K y$. Therefore $\lim_{y\to 0}D_y u(t,x,y)=0$ for $t\in [\delta,T],\ x\in\R^{N}$.
			%	\end{proof}
		%}	

	We now show that for a compactly supported initial datum $u_0$ the semigroup solution $u=e^{t\mathcal L}u_0$ has the regularity above.
	\begin{lem}\label{regularity of semigroup}
		Let  $u_0\in C_c\left(\ov{\HS}\right)$ and $u=e^{t\mathcal L}u_0$. Then   
		$$u\in C^{1,2}\left(]0,\infty[\times\ov{\R^{N+1}_+}\right)\cap C\left([0,\infty[\times\ov{\R^{N+1}_+}\right).$$
	\end{lem}
	\begin{proof}
		Since the semigroup solution $u$ is in particular  a weak solution of the Cauchy Problem \eqref{CP L}, the regularity   $u\in C^{1,2}\left(]0,\infty[\times \ov\HS\right)$ follows from the results in    It remains to prove that $u\in C\left([0,\infty[\times\ov{\R^{N+1}_+}\right)$.
		Let us assume, preliminarily, $u_0\in C^{\infty}_{c}\left(\ov\HS\right)$. 
		We use Theorem \ref{Teorema generazione} with the Lebesgue measure and a sufficiently large $p$.  We have   $$C_c^\infty\left(\ov\HS\right) \subset D(\mathcal L)\hookrightarrow W^{2,p}\left(\HS\right) \hookrightarrow C_0(\ov\HS),$$ by Sobolev embedding. Then  $u(t,\cdot)\rightarrow u_0$ in $W^{2,p}\left(\HS\right)$ as $t\rightarrow 0$ and so uniformly.

		%Then, recalling Theorem \ref{Teorema generazione}, we have   $C_c^\infty\left(\ov\HS\right)\hookrightarrow D(\mathcal L^n)\hookrightarrow W^{2n,2}\left(\HS\right)$,  and so fixing a sufficiently large $n\in\N$ we get $D(\mathcal L^n)\hookrightarrow C(\HS)$. 
		%Let us consider now the semigroup in $D(\mathcal L^n)$; we have obviously $u(t,\cdot)\in D(L^n)\subseteq C^2\left(\Omega\right)$ and since $u$ is a solution of $\frac{d}{dt}u(t,\cdot)=Lu(t,\cdot)$ in $D(L^n)$, the embedding $D(L^n)\hookrightarrow C^2(\Omega)$ yields that the time derivative is a classical derivative and  we have  also $\frac{d}{dt}u(t,x)=Lu(t,x)$ pointwise. This proves  $u\in C^{1,2}\left(]0,\infty[\times \Omega\right)$. Analogously $u(t,\cdot)\rightarrow f$ in $D(L^n)$ as $t\rightarrow 0$ and so pointwise   and this implies  $u\in C\left([0,\infty[\times\Omega\right)$.
		
		Since by \cite[Proposition 6.1]{Negro-Spina-SingularKernel-Lower} the semigroup $e^{t\mathcal L}$ is uniformly bounded in 
		$L^{\infty}\left(\HS\right)$, by approximation we get $e^{t\mathcal L} u_0 \to u_0$ as $t \to 0$ for any $u_0$ continuous and with a compact support. 
	\end{proof}
	
	We also need a maximum principle for our operator.
	
	\begin{prop}\label{maximum}
		Let  $R>0$, $u\in C^{1,2}\left(]0,T]\times\ov{\R^{N+1}_+}\right)\cap C\left([0,T]\times\ov{\R^{N+1}_+}\right)$,   be such that
		\begin{align*}	\begin{cases}
				\partial_t u-\mathcal Lu\leq 0 & 0<t\leq T, \, z \in \HS\\[1ex]
				u(0,z) \leq 0 & z\in\ov{\R^{N+1}_+}\\[1ex]
				D_y u(t,x,0)=0&0<t \leq T,\ x\in\R^{N}\\[1ex]
				u(t,z)\leq 0&0 \leq t \leq T,\ |z|\geq R.
			\end{cases}
		\end{align*}
		Then $u\leq 0$ in $[0,T]\times\ov{\R^{N+1}_+}$.
	\end{prop}
	\begin{proof}
		Let $\eps>0$, $v=u-\eps t$. Then $\partial_t v-\mathcal Lv\leq-\eps<0$ and $v$ satisfies all the other assumptions as $u$.
		Assume $v>0$ somewhere. Then $v$ has a positive maximum at some point $P_0=(t_0,x_0,y_0)$ with $t_0>0$.  Moreover $\partial_t v(P_0)\geq 0$. We now distinguish two cases. If $P_0$ is an inner point i.e. $y_0>0$, then $D_yv(P_0)=0$ and $\mathcal L v (P_0)\leq 0$ by the ellipticity assumption.   If $P_0$ is on the boundary of $\HS$, i.e. $y_0=0$, then $\Delta_xv(P_0)\leq 0$ and  $\nabla_xD_yv(P_0)=0$, by the boundary condition. Moreover 
		\begin{align*}
			\lim_{y\to 0}\frac{D_yv(t_0,x_0,y)}{y}=D_{yy} v(P_0)
		\end{align*}
		which implies, since $c+1>0$,
		\begin{align*}
			\lim_{y\to 0}\left(D_{yy}v(t_0,x_0,y)+c\frac{D_yv(t_0,x_0,y)}{y}\right)=(c+1)D_{yy}v(P_0)\leq 0.
		\end{align*}
		Therefore we have in any case $\mathcal L  v(P_0)\leq 0$. This implies $\partial_t v(P_0)-\mathcal L  v(P_0)\geq 0$ and this is in contradiction with $\partial_t v-\mathcal Lv<0$.
		It follows that $v\leq 0$ and so $u\leq \eps t$ in  $[0,T]\times\ov{\R^{N+1}_+}$. Letting $\eps \to 0$, the claim  follows.
	\end{proof}
	
	We can now prove uniqueness under a pointwise growth condition. We use a barrier as in the case of the Laplacian. In the next section we shall prove a similar result under a weaker integral growth condtion, see Theorem \ref{uniqueness tacklind}.

	\begin{teo}\label{Tychonov}
		Let  $u\in C^{1,2}\left(]0,T]\times\ov{\R^{N+1}_+}\right)\cap C\left([0,T]\times\ov{\R^{N+1}_+}\right)$ be a solution of the Cauchy problem 
		\begin{align*}	\begin{cases}
				\partial_t u=\mathcal Lu & 0<t \leq T,\ z\in \ov{\R^{N+1}_+}\\[1ex]
				u(0,z) =0 & z\in\R^{N+1}_+.\\[1ex]
				%\ds\lim_{y\to 0}y^{c}\, D_y u(t,x,y)=0,&x\in\R^{N}.
			\end{cases}
		\end{align*}
		If  $|u(t,z)|\leq C e^{\beta|z|^2}$  for some positive constant $\beta,C>0$, 
		then $u \equiv 0$.
	\end{teo}
	\begin{proof}
		We fix  $\alpha>0$ and let  $\kappa, \delta >0$ to be chosen later. Set $$v(t,x,y)=\frac{1}{(\delta-t)^\alpha} \exp\left\{\frac{\kappa}{\delta-t} (|x|^2+y^2)\right\}.$$ Then for $0 \leq t <\delta$
		\begin{align*}
			&\partial_tv= \left (\frac{\alpha}{\delta-t}+\frac{\kappa}{(\delta-t)^2}(|x|^2+y^2) \right )v, \quad D_{x_i}v=\frac{2\kappa x_i}{\delta-t}v, \quad D_yv=\frac{2\kappa y}{\delta-t}v \\
			&D_{x_i x_i}v=\left (\frac{4\kappa^2 x_i^2}{(\delta-t)^2}+\frac{2\kappa}{\delta-t} \right )v, \quad D_{yy}v=\left (\frac{4\kappa^2 y^2}{(\delta-t)^2}+\frac{2\kappa}{\delta-t} \right )v, \quad D_{x_iy}v=\frac{4\kappa^2 x_i y}{\delta-t} v.
		\end{align*}
		Then, choosing  $\kappa$  sufficiently small, we get 
		\begin{align*}
			v_t-\mathcal L v= \left (\frac{\alpha}{\delta-t}+\frac{\kappa-4 \kappa^2}{(\delta-t)^2}(|x|^2+y^2)-\frac{2\kappa (N+1+c)}{\delta-t}-\frac{4\kappa^2}{\delta-t}(a\cdot x)y \right )v \geq 0.
		\end{align*}
		Let us choose $\delta=\frac{k}{2\beta}$. Given $\eps>0$, the function $w=u-\eps v$ goes to $-\infty$, as $|(x,y)|$ goes to infinity, uniformly in $\left[0,\frac{\delta}{2}\right]$. Then
		\begin{align*}	\begin{cases}
				\partial_t w-\mathcal Lw\leq 0 & 0<t \leq \frac{\delta}{2}, \, z\in  \ov{\R^{N+1}_+}\\[1ex]
				w(0,z) <0 & z\in\R^{N+1}_+\\[1ex]
				 D_y w(t,x,y)=0&x\in\R^{N}\\[1ex]
				\ds\lim_{|z|\to +\infty} w(t,x,y)=-\infty &{\rm uniformly\ for\ } t\in \left [0,\frac{\delta}{2}\right].
			\end{cases}
		\end{align*}
		Proposition \ref{maximum} yields $w\leq 0$ everywhere and, letting $\eps$ to $0$, $u\leq 0$ in $\left[0,\frac{\delta}{2}\right]\times \ov{\R^{N+1}_+}$. By arguing similarly for $-u$, we deduce $u= 0$ in $\left[0,\frac{\delta}{2}\right]\times \ov{\R^{N+1}_+}$. By repeating the same arguments in steps of lenght $\frac{\delta}{2}$, we get $u=0$ in $[0,T]\times \HS$.
	\end{proof}

	\section{Mean value inequalities and uniqueness under integral growth conditions}

We  need  mean value inequalities and interior estimates for solutions. To this aim  for any  $r> 0$ and any  $(t_0,z_0)\in[0,T]\times\ov{\R^{N+1}_+}$, we   consider the   parabolic cylinders
	\begin{align*}
		\mathcal I\left((t_0,z_0),r\right):=]t_0-r^2,t_0[\times Q(z_0,r).
	\end{align*} 
	\begin{lem} \label{stima-sup}
		Let  $u\in C^{1,2}\left(]0,T]\times\ov{\R^{N+1}_+}\right)\cap C\left([0,T]\times\ov{\R^{N+1}_+}\right)$ be a solution of the  Cauchy problem 
		\begin{align*}	\begin{cases}
				\partial_t u=\mathcal Lu, & 0<t \leq T, \, z\in {\R^{N+1}_+}\\[1ex]
				u(0,z) =0, & z\in\ov{\R^{N+1}_+}\\[1ex]
				%\ds\lim_{y\to 0}y^{c}\, D_y u(t,x,y)=0,&x\in\R^{N}.
			\end{cases}
		\end{align*}
 and let us extend $u(s,\cdot)$ to $0$ for negative times $s<0$.
		For $r> 0$, $(t_0,z_0)\in[0,T]\times\ov{\R^{N+1}_+}$ and $0<p\leq 2$ we have
		$$\sup_{\mathcal I\left((t_0,z_0),\frac{r}{2}\right)}  |u|\leq C \Big(\fint_{\mathcal I\left((t_0,z_0),r\right)} |u(s,z_2)|^p\, ds\,d\mu(z_2)\Big)^{\frac 1 p}$$ 
		for some $C>0$ independent of $u$, $r$, $(t_0,z_0)$. 
	\end{lem}	
	
	\begin{proof} 
		Let $\tilde u$ the extension by $0$ of $u(s,\cdot)$ for $s<0$. As in \cite[Section 1, pp. 620-621]{Aronson},  $\tilde u$ is a  weak solution in $ ]-\infty,T]\times \HS$ of  
		\begin{align*}
			\begin{cases}
				\partial_t u=\mathcal Lu, & -\infty <t \leq T,\, z\in \HS\\[1ex]
				\ds\lim_{y\to 0}y^{c}\, D_y u(t,x,y)=0,&x\in\R^{N}.
			\end{cases}
		\end{align*}
		Let $r> 0$, $(t_0,z_0)\in[0,T]\times\ov{\R^{N+1}_+}$. We then apply  \cite[Lemma 4.3]{dong2020neumann} to $\tilde u$  thus getting
		$$\sup_{\mathcal I\left((t_0,z_0),\frac{r}{2}\right)}  |u|\leq C \Big(\fint_{\mathcal I\left((t_0,z),r\right)} |u(s,z_2)|^2\, ds\,d\mu(z_2)\Big)^{\frac 1 2}.$$ 
		By a standard iteration argument (see e.g. \cite[Theorem 5.4 and  Remark 2, pp. 80-82]{Giaquinta-book}), the previous inequality extends to any $0<p<2$
		$$\sup_{\mathcal I\left((t_0,z_0),\frac{r}{2}\right)}  |u|\leq C \Big(\fint_{\mathcal I\left((t_0,z),r\right)} |u(s,z_2)|^p\, ds\,d\mu(z_2)\Big)^{\frac 1 p}.$$ 
	\end{proof}

\begin{os}\label{equiv Lp norms}
Lemma	\ref{stima-sup} implies that all $L^p$-norms of the solutions for $0<p \leq \infty$ are equivalent on cylinders.
%, with contsants independent. Indeed, for any $0<p_1<p_2\leq 2$, one has by H\"older inequality 
%\begin{align*}
% \Big(\fint_{\mathcal I\left((t_0,z_0),r\right)} |u(s,z_2)|^{p_1}y_2^c\, ds\,dz_2\Big)^{\frac 1 {p_1}}\leq  \Big(\fint_{\mathcal I\left((t_0,z_0),r\right)} |u(s,z_2)|^{p_2}y_2^c\, %ds\,dz_2\Big)^{\frac 1 {p_2}}
%\end{align*}
%and, by the mean value inequality of Lemma	\ref{stima-sup}, one has  for some $C>0$, 
%\begin{align*}
%	\Big(\fint_{\mathcal I\left((t_0,z_0),r\right)} |u(s,z_2)|^{p_2}y_2^c\, ds\,dz_2\Big)^{\frac 1 {p_2}}\leq C \Big(\fint_{\mathcal I\left((t_0,z_0),2r\right)} |u(s,z_2)|^{p_1}y_2^c\, %ds\,dz_2\Big)^{\frac 1 {p_1}}
%\end{align*}

\end{os}
We  need also  Schauder estimates up to the boundary for local solutions of the equation $\partial_t u-\mathcal Lu=0$. 
 To this aim  for any  $0<\theta\leq 1$ and any parabolic cylinder $\mathcal I$ we define the  H\"older norm of a function $u$ in $\mathcal I$ as
\begin{align*}
		\|u\|_{C^{\frac\delta 2,\delta}\left(\mathcal I\right)}&:=\|u\|_{L^\infty\left(\mathcal I\right)}+[u]_{C^{\frac\delta 2,\delta}\left(\mathcal I\right)},\\[1.5ex]
		[u]_{C^{\frac\delta 2,\delta}\left(\mathcal I\right)}&:=\sup_{\substack{(t_1,z_1),(t_2,z_2)\in\mathcal I\\(t_1,z_1)\neq (t_2,z_2)}}
		\frac{|u(t_1,z_1)-u(t_2,z_2)|}{|t_1-t_2|^{\frac\delta 2}+|z_1-z_2|^\delta}.
\end{align*}
We  also define 
$$
\|u\|_{C^{1+\frac\delta 2,2+\delta}\left(\mathcal I\right)}=\|u\|_{C^{\frac\delta 2,\delta}\left(\mathcal I\right)}+
\|D_tu\|_{C^{\frac\delta 2,\delta}\left(\mathcal I\right)}+\|\nabla u\|_{C^{\frac\delta 2,\delta}\left(\mathcal I\right)}+
\|D^2u\|_{C^{\frac\delta 2,\delta}\left(\mathcal I\right)}.
$$

In the next lemma we consider parabolic cylinders contained in $]0,T]\times\ov{\R^{N+1}_+}$.
\begin{lem} \label{stima-Schauder}
 	Let  $(t_0,z_0)\in]0,T]\times\ov{\R^{N+1}_+}$ and consider a  cylinder $\mathcal I_r:=\mathcal I\left((t_0,z_0),r\right) \subset ]0,T]\times\ov{\R^{N+1}_+}$. Let $u\in C^{1,2}\left( \overline {\mathcal I_r}\right)$ be a solution of the  equation 
	\begin{align}\label{Schauder eq 1}
				\partial_t u(t,z)=\mathcal Lu(t,z), \qquad  (t,z)\in \mathcal I_r
	\end{align}
	Then for any $0<\delta\leq 1$, $u$ belongs to $C^{1+\frac\delta 2,2+\delta}\left(\mathcal I_\frac{r}2\right)$ and 
	\begin{align*}
		\|u\|_{C^{1+\frac\delta 2,2+\delta}\left(\mathcal I_{\frac r 2}\right)}+\left\|\frac{D_yu}y\right\|_{C^{\frac\delta 2,\delta}\left(\mathcal I_{\frac r 2}\right)}\leq C \Big(\fint_{\mathcal I_r} |u(s,z_2)|^2\, ds\,d\mu(z_2)\Big)^{\frac 1 2}
	\end{align*}
	for some positive constant $C$ depending on $(t_0,z_0)$ and $r$ but  independent of $u$. 
	
\end{lem}	
\begin{proof}
	We use  the interior estimates \cite[Proposition 4.4]{dong2020neumann} and the Caccioppoli inequalities \cite[ Lemma 4.2]{dong2020neumann} (which actually apply more generally to weak solutions of equation \eqref{Schauder eq 1}) thus obtaining, 	for some positive constant $C$,
	\begin{align} \label{eq1}
		\|u\|_{C^{\frac\delta 2,\delta}\left(\mathcal I_{\frac r 2}\right)}+\|\nabla u\|_{C^{\frac\delta 2,\delta}\left(\mathcal I_{\frac r 2}\right)}+\|\partial_t u\|_{C^{\frac\delta 2,\delta}\left(\mathcal I_{\frac r 2}\right)}\leq C \Big(\fint_{\mathcal I_r} |u(s,z_2)|^2\, ds\,d\mu(z_2)\Big)^{\frac 1 2}.
	\end{align}
Since the coefficients of the equation do not depend on $t$ and $x$, a standard argument based on difference quotients as in \cite[Proposition 4.4]{dong2020neumann}, proves that one can  formally differentiate the equation with respect to $t$ and $x$ thus proving that any derivative $D^{\alpha}_tD^{\beta}_xu$  is a weak solution of \eqref{Schauder eq 1}. 	 In  particular by applying the previous estimate to $\nabla_xu$ and the Caccioppoli inequalities \cite[ Lemma 4.2]{dong2020neumann} we also obtain
\begin{align*}
	\|D^2_xu\|_{C^{\frac\delta 2,\delta}\left(\mathcal I_{\frac r 2}\right)}+\|\nabla_x D_y u\|_{C^{\frac\delta 2,\delta}\left(\mathcal I_{\frac r 2}\right)}\leq C \Big(\fint_{\mathcal I_r} |u(s,z_2)|^2\, ds\,d\mu(z_2)\Big)^{\frac 1 2}.
\end{align*}
It remains to bound $D_{yy}u$, and $\frac{D_yu}{y}$. Using  equation \eqref{Schauder eq 1} we can write 
\begin{align*}
	B_yu:=D_{yy}u+c\frac{D_yu}y=\partial_tu-\left(\Delta_x u+2a\cdot \nabla_xD_yu \right)
\end{align*}
which, using the previous estimates, implies
\begin{align}\label{Schauder eq 2}
	\|B_yu\|_{C^{\frac\delta 2,\delta}\left(\mathcal I_{\frac r 2}\right)}\leq C \Big(\fint_{\mathcal I_r} |u(s,z_2)|^2\, ds\,d\mu(z_2)\Big)^{\frac 1 2}.
\end{align}
If $y_0-r >0$, then $y \geq \frac r 2$ in $\mathcal I_{\frac r 2}$ and the estimate of the H\"older norm of $\frac{D_y u}{y}$ follows from \eqref{eq1}.  Together with \eqref{Schauder eq 2} this gives the estimate for the H\"older norm of $D_{yy} u$ in $\mathcal I_{\frac r 2}$. 

Finally, if $y_0-r \leq 0$, then $(x,0) \in \mathcal I_r$ when $(x,y) \in \mathcal I_r$.
Since $D_{y}\left(y^{c}D_yu\right)=y^{c}B_yu$ and $\ds\lim_{y\to 0}y^cD_yu=0$, we have
\begin{align*}
	y^{c}D_yu(t,x,y)=\int_0^y s^cB_yu(t,x,s)ds\stackrel{{\tiny \tau y=s}}{=}y^{c+1}\int_0^1 \tau^cB_yu(t,x,\tau y)d\tau,\qquad (t,x,y)\in \mathcal I_r
\end{align*}
which implies
\begin{align*}
	\frac{D_yu(t,x,y)}{y}=\int_0^1 \tau^cB_yu(t,x,\tau y)d\tau,\qquad (t,x,y)\in \mathcal I_r.
\end{align*}

Using \eqref{Schauder eq 2}, the last equation implies
\begin{align*}
\left\|\frac{D_yu}y\right\|_{C^{\frac\delta 2,\delta}\left(\mathcal I_{\frac r 2}\right)}&\leq \int_0^1 \tau^c\left\|B_yu(\cdot,\cdot,\tau \cdot)\right\|_{C^{\frac\delta 2,\delta}\left(\mathcal I_{\frac r 2}\right)}d\tau
\leq C\left (\int_0^1 \tau^{c} \,d\tau \right )\Big(\fint_{\mathcal I_r} |u(s,z_2)|^2\, ds\,d\mu(z_2)\Big)^{\frac 1 2}.
\end{align*}
By difference, the last inequality implies the estimate for $D_{yy}u$.
\end{proof}

	The previous mean value inequalities allow to prove  the following  generalization of the uniqueness Theorem \ref{Tychonov} which allows to replace pointwise  growth conditions with  some integral  growth conditions.

\begin{lem}\label{equiv growth condition}
	Let  $ u\in C^{1,2}\left(]0,T]\times\ov{\R^{N+1}_+}\right)\cap C\left([0,T]\times\ov{\R^{N+1}_+}\right)$ be a  solution of the  Cauchy problem 
	\begin{align*}	\begin{cases}
			\partial_t u=\mathcal Lu, & 0<t \leq T,\, z\in \ov{\R^{N+1}_+}\\[1ex]
			u(0,z) =0, & z\in\ov{\R^{N+1}_+}.\\[1ex]
			%\ds\lim_{y\to 0}y^{c}\, D_y u(t,x,y)=0,&x\in\R^{N}.
		\end{cases}
	\end{align*}
Let $0<p<\infty$. The following  growth conditions are equivalent:
\begin{itemize}
	\item[(i)] $\ds |u(t,z)|\leq C e^{\beta|z|^2}$  for some positive constants $\beta,C>0$; \smallskip 
	\item[(ii)] $\ds \int_0^T\int_{\HS}e^{-\beta|z|^2}|u(t,z)|^p\,dt\, d\mu(z)\leq C$  for   some    positive constants $\beta,C>0$;\smallskip 
\item[(iii)] $\ds\int_0^T\int_{Q(0,r)}|u(t,z)|^p\,dt\, d\mu(z)\leq Ce^{\beta r^2},\ r>0$, for some  positive constants   $\beta,C>0$.

\end{itemize}
\end{lem}
\begin{proof} We start by observing that (i) implies (ii) since, taking $\beta'>\beta p$, one has
\begin{align*}
	\int_0^T\int_{\HS}e^{-\beta' |z|^2}|u(t,z)|^p\,dt\, d\mu(z)\leq C^p \int_0^T\int_{\HS}e^{-(\beta'-\beta p) |z|^2}\,dt\, d\mu(z)<+\infty.
\end{align*}	
On the other hand  (ii) implies   (iii) since
	\begin{align*}
	\int_0^T\int_{Q(0,r)}|u(t,z)|^p\,dt\, d\mu(z)\leq e^{\beta r^2} \int_0^T\int_{\HS}e^{-\beta |z|^2}|u(t,z)|^p\,dt\, d\mu(z).
	\end{align*} 
 It is then  sufficient to prove that (iii) implies (i). By remark \ref{equiv Lp norms} and the doubling condition of Lemma \ref{Misura palle}, we can suppose without any loss of generality $p=1$. 	

Let us fix $t\in [0,T[$ and $z\in\HS$ and choose $t_0\in[0,T[$ such that $t_0-\frac{(1+|z|)^2}{4}<t<t_0$; in particular $(t,z)\in \mathcal I\left((t_0,z),\frac{1+|z|}2\right)$. We  observe that 
 \begin{align*}
\mathcal I\Big((t_0,z),1+|z|\Big)\subseteq \left(t_0-(1+|z|)^2,T\right)\times Q\Big(0,1+2|z|\Big)
 \end{align*}
and that by Lemma \ref{Misura palle} we have
\begin{align*}
	\left|\mathcal I\Big((t_0,z),1+|z|\Big)\right|=(1+|z|)^2V(z,1+|z|)\simeq (1+|z|)^{2+N+1+c}\geq 1. 
\end{align*}
By extending $u$ to $0$ for  negative times and by using (i) with $r=1+2|z|$ and the previous relations we then obtain
 \begin{align*}
 \fint_{\mathcal I\left((t_0,z),1+|z|\right)}| u(s,z_2)|ds\,d\nu(z_2) &=\frac{1}{	\left|\mathcal I\Big((t_0,z),1+|z|\Big)\right|}\int_{\mathcal I\left((t_0,z),1+|z|\right)} |u(s,z_2)|\, ds\,d\nu(z_2)\\[2ex]
 &\leq  C\int_0^{T}\int_{Q\left(0,1+2|z|\right)} |u(s,z_2)|\, ds\,d\mu(z_2)\leq 
  C  e^{\beta |z|^2}.
 \end{align*}
 By applying Lemma \ref{stima-sup} in $\mathcal I\left((t_0,z),1+|z|\right)$,
 we then deduce
 $$|u(t,z)|\leq   C e^{\beta|z|^2}, \qquad  0\leq  t<T.$$
% Then $u$ satisfies the growth assumptions of Theorem  \ref{Tychonov} in the strip $[0,T[\times \ov{\R^{N+1}_+}$ and then $u=0$. 
\end{proof}

\begin{teo}\label{uniqueness tacklind}
	Let  $ u\in C^{1,2}\left(]0,T]\times\ov{\R^{N+1}_+}\right)\cap C\left([0,T]\times\ov{\R^{N+1}_+}\right)$ be a  solution of the  Cauchy problem 
	\begin{align*}	\begin{cases}
			\partial_t u=\mathcal Lu, &0<t \leq T,\, z\in \ov{\R^{N+1}_+}\\[1ex]
			u(0,z) =0, & z\in\ov{\R^{N+1}_+}\\[1ex]
			%\ds\lim_{y\to 0}y^{c}\, D_y u(t,x,y)=0,&x\in\R^{N}.
		\end{cases}
	\end{align*}
	satisfying one of the equivalent growth conditions of Lemma \ref{equiv growth condition}.
	Then $u=0$ in $[0,T]\times\ov{\R^{N+1}_+}$.
\end{teo}
\begin{proof} This follows follows from Lemma  \ref{equiv growth condition} and  Theorem  \ref{Tychonov}. 
\end{proof}

	\section{Uniqueness of positive solutions}
We now prove  a generalization of a result proved by Widder \cite{Widder}  concerning the uniqueness of non-negative solutions for the heat equation. We first  prove  the minimality of the semigroup  $e^{t\mathcal L}u_0$ among the set of all positive  solutions of the Cauchy problem \eqref{CP L}.
	\begin{prop}\label{domination}
		Let  $0 \leq u\in C^{1,2}\left(]0,T]\times\ov{\R^{N+1}_+}\right)\cap C\left([0,T]\times\ov{\R^{N+1}_+}\right)$ be a solution of the  equation
		\begin{align*}	%\begin{cases}
			\partial_t u=\mathcal Lu\quad & {\rm in}\ ]0,T]\times\ov{\R^{N+1}_+}.\\[1ex]
			%u(0,z) =0, & z\in\R^{N+1}_+,\\[1ex]
			%\ds\lim_{y\to 0}y^{c}\, D_y u(t,x,y)=0,&x\in\R^{N}.
			%\end{cases}
		\end{align*}
		Then 
		$$v(t,z)=\int_{\HS}p_{\mathcal L}(t,z,z_2)u(0,z_2)\, d\mu(z_2)\leq u(t,z),\qquad (t,z)\in[0,T]\times\ov{\R^{N+1}_+}.$$
		Moreover  $v\in C^{1,2}\left(]0,T]\times\ov{\R^{N+1}_+}\right)\cap C\left([0,T]\times\ov{\R^{N+1}_+}\right)$, $v_t=\mathcal L v$ in $]0,T]\times\ov{\R^{N+1}_+}$ and $v(0,z)=u(0,z)$.
	\end{prop}
	\begin{proof} Set $u_0(\cdot)=u(0,\cdot)$. For a fixed $r>0$ let $\eta_r$ be a smooth function such that $0\leq \eta_r\leq 1$, $\eta_r=1$ on $Q(0,r)$ and with support contained in $Q(0,2r)$. Since $\eta_r u_0\in C_c\left(\ov{\HS}\right)$ we may apply Lemma \ref{regularity of semigroup} to obtain that 
		$$v_r(t,z):=e^{t\mathcal L}(\eta_r u_0)(t,z)=\int_{\HS}p_{\mathcal L}(t,z,z_2)\eta_r(z_2)u(0,z_2)\, d\mu(z_2)$$
belongs to $ C^{1,2}\left(]0,T]\times\ov{\R^{N+1}_+}\right)\cap C\left([0,T]\times\ov{\R^{N+1}_+}\right)$.
		It immediately follows that $v_r\to v$ as $r$ goes to infinity by monotone convergence. By  the   kernel estimates in Proposition \ref{kernel tilde} and by Lemma \ref{stima V locale} one has, for some positive constants  $C$ depending on $r$ and $T$ (and which may vary in each occurrence), 
		\begin{align*}
			p_{{\mathcal  L}}(t,z,z_2)
			&\leq \frac{C}{V\left(z_2,\sqrt t\right)}\,\exp\left(-\dfrac{|z-z_2|^2}{kt}\right)\\[1ex]&
			\leq C t^{-\frac{N+1+c^+}{2}}\,\exp\left(-\dfrac{|z-z_2|^2}{kt}\right),\qquad 0<t\leq T,\quad |z_2|\leq 2r .
		\end{align*}	
	 It follows that, if $|z_2|\leq 2r$, then by elementary calculation
			\begin{align}\label{domination eq 1}
				\sup_{0< t\leq T}p_{\mathcal L}(t,z,z_2)\leq C	\sup_{t>0}\Big[t^{-\frac{N+1+c^+}{2}}\,\exp\left(-\dfrac{|z-z_2|^2}{kt}\right)\Big]= \frac{C}{|z-z_2|^{N+1+c^+}}.
			\end{align} 
			In particular if $|z|>4r$ then $|z-z_2|\geq  \frac{|z|}{2}$ and   
		\begin{align}\label{domination eq 2}
			\sup_{0< t\leq T}p_{\mathcal L}(t,z,z_2)\leq \frac{C}{ |z|^{N+1+c^+}}.
		\end{align} 
			Therefore
 			$$v_r(t,z)= \int_{Q(0,2r)}p_{\mathcal L}(t,z,z_2)\eta_r(z_2)u(0,z_2)\, d\mu(z_2)$$
	 tends to $0$ as $|z|$ goes to infinity, uniformly in $[0,T]$. Given $\eps>0$, we can find some $R>0$ such that
		\begin{align*}
			w(t,z):=v_r(t,z)-\eps -u(t,z)\leq 0,\qquad t\in [0,T],\ |z|\geq R.
		\end{align*}  
		Moreover $\partial_t w-\mathcal L w=0$,  $w(0,z)\leq 0$.
		Then by applying Proposition \ref{maximum}  we get $w \leq 0$ in $[0,T]\times\ov{\HS}$. Letting $\eps \to 0$ and then $r \to \infty$ we obtain $v(t,z) \leq u(t,z)$ and, in particular, $v$ is finite everywhere.

	Let us finally  prove that $v\in C^{1,2}\left(]0,T]\times\ov{\R^{N+1}_+}\right)\cap C\left([0,T]\times\ov{\R^{N+1}_+}\right)$ and solves the same Cauchy problem. The continuity at $t=0$ follows since, by the previous step we have
	\begin{align}\label{domination eq 3}
		v_r\leq v\leq u
	\end{align}
which implies $\lim_{t\to 0}v(t,z)=u_0(z)$ in $Q_r$. Let us now fix $(t_0,z_0)\in ]0,T]\times \HS$ and  a sufficiently small $R_0>0$ such that 
\begin{align*}
	\mathcal I_{R_0}:=\mathcal I\left((t_0,z_0),R_0\right)\subseteq ]0,T]\times  \HS.
\end{align*} We observe that, from \eqref{domination eq 3}, we also have $v_r\to v$ in $L^2\left(\mathcal I_{R_0},y^cdtdz\right)$ by dominated convergence with $u_{\vert_{\mathcal I_{R_0}}}$.  By Lemma \ref{stima-Schauder},  we have for $0<\delta\leq 1$ and  any $r_1,r_2>0$
\begin{align*}
		\|v_{r_1}-v_{r_2}\|_{C^{1+\frac\delta 2,2+\delta}\left(\mathcal I_{\frac {R_0} 2}\right)}\leq C \|v_{r_1}-v_{r_2}\|_{L^2\left(\mathcal I_{R_0},y^cdtdz\right)}
\end{align*}
which shows that $\left(v_r\right)_{r>0}$ satisfies the Cauchy condition  in $C^{1+\frac\delta 2,2+\delta}\left(\mathcal I_{\frac {R_0} 2}\right)$. This implies that $v_r\to v$ in $C^{1+\frac\delta 2,2+\delta}\left(\mathcal I_{\frac {R_0} 2}\right)$. In particular  $v\in C^{1,2}(]0,T]\times\ov{\R^{N+1}_+})$ and satisfies  the equation $\partial_t v=\mathcal L v$. 
	\end{proof}

	\begin{teo}\label{uniqueness}
		Let  $0\leq u\in C^{1,2}\left(]0,T]\times\ov{\R^{N+1}_+}\right)\cap C\left([0,T]\times\ov{\R^{N+1}_+}\right)$ be a solution of the  Cauchy problem 
		\begin{align*}	\begin{cases}
				\partial_t u=\mathcal Lu, &0<t \leq T,\, z\in \ov{\R^{N+1}_+}\\[1ex]
				u(0,z) =0, & z\in\ov{\R^{N+1}_+}.\\[1ex]
				%\ds\lim_{y\to 0}y^{c}\, D_y u(t,x,y)=0,&x\in\R^{N}.
			\end{cases}
		\end{align*}
		Then $u=0$ in $[0,T]\times\ov{\R^{N+1}_+}$.
	\end{teo}
	
	\begin{proof} 
		Let  $\eps>0$, $0\leq s<T-\eps$ and fix $\ov{z}\in \R^{N+1}_+$. By Proposition  \ref{domination} 
		$$\int_{\HS} p_{\mathcal L}(T-s,\ov z,z_2)u(s,z_2)\, d\mu(z_2) \leq u(T,\ov{z}).$$
		By the lower estimates of Proposition \ref{kernel tilde}, we get
		$$\frac{C}{V(\ov{z}, \sqrt{\epsilon})}\int_{\HS} e^{-\frac{|z_2-\ov{z}|^2}{k\epsilon}}u(s,z_2)\, d\mu(z_2) \leq \int_{\HS} p_{\mathcal L}(T-s,\ov z,z_2)u(s,z_2)\,d\mu(z_2) \leq u(T,\ov{z})$$ for all $s<T-\eps$.
		It follows that for all $s<T-\eps$ and suitable $C_\epsilon, a_\eps>0$
		$$C_\eps\int_{\HS} e^{-a_\eps|z_2|^2}u(s,z_2)\, d\mu(z_2)\leq u(T,\ov{z}).$$ 
		After integrating with respect to $s$, we see that $u$ satisfies condition (iii) in Lemma \ref{equiv growth condition} in the strip $[0,T-\eps]\times \ov{\R^{N+1}_+}$  and, by Theorem \ref{uniqueness tacklind},  $u=0$ in  $[0,T-\eps]\times\ov{\R^{N+1}_+}$, hence in $[0,T]\times\ov{\R^{N+1}_+}$ by the arbitrarity of $\eps$.  
		\end{proof}

	As consequence we get the following representation of positive solution trough the integral kernel.
	\begin{teo}\label{Uniqueness Theorem}
		Let  $0\leq u\in C^{1,2}(]0,T]\times\ov{\R^{N+1}_+})\cap C([0,T]\times\ov{\R^{N+1}_+})$ be a solution of the  Cauchy problem 
		\begin{align*}	\begin{cases}
				\partial_t u=\mathcal Lu, &0<t \leq T,\, z\in \ov{\R^{N+1}_+}\\[1ex]
				u(0,z) =u_0(z), & z\in\ov{\R^{N+1}_+}.\\[1ex]
				%\ds\lim_{y\to 0}y^{c}\, D_y u(t,x,y)=0,&x\in\R^{N}.
			\end{cases}
		\end{align*}
		Then 
		\begin{align}\label{Widder uniq eq 1}
			u(t,z)=\int_{\HS}p_{\mathcal L}\left(t,z,z_2 \right)u_0(z_2)\,d\mu(z_2),\qquad 0 \leq t \leq T, \, z\in \ov{\R^{N+1}_+}.
		\end{align}
		%Furthermore  a necessary and sufficient condition for the existence of $u$ is
		%\begin{align*}
		%\int_{\HS}p_{\mathcal L}\left(t,z_1,z_2 \right)u_0(z_2)\,y_2^cdz_2\in L^2\Big([0,T];L^2_{loc}(\HS)\Big).
		%\end{align*}
	\end{teo}
	\begin{proof}
		Let us consider 
		\begin{align*}
			v(t,z)=\int_{\HS}p_{\mathcal L}\left(t,z_1,z_2 \right)u_0(z_2)\,d\mu(z_2).
		\end{align*}
		%\Red{By ??? $v\in C^{1,2}(]0,T]\times\ov{\R^{N+1}_+})\cap C([0,T]\times\ov{\R^{N+1}_+})$ and solves the same previous Cauchy problem. ******Mi pare vada provato. La proposizione \ref{regularity of semigroup} prevede $u_0\in L^2_c$. Mi pare possa seguire dalla proposizione \ref{domination} ma va chiarita la convergenza in $C^{1,2}$******\\}
		By Proposition \ref{domination}, $v\in C^{1,2}(]0,T]\times\ov{\R^{N+1}_+})\cap C([0,T]\times\ov{\R^{N+1}_+})$, solves the same Cauchy problem as $u$  and $u(t,z)\geq v(t,z)$. Therefore $w=u-z\geq 0$ solves the Cauchy problem with null initial datum and, by Theorem \ref{uniqueness}, $w=0$, that is $u=v$.
	\end{proof}

	\section{Harnack inequality}\label{Section harnack}
	A pointwise Harnack inequality for non-negative  solutions of the equation  
$\partial_t u=\mathcal Lu$
	can be derived from the uniqueness of the non-negative solutions guaranteed by Theorem \ref{Uniqueness Theorem}  and from the upper and lower estimates \eqref{up kernel measure} of the heat kernel $p_{\mathcal L}$ of $\mathcal L$. 
	We state the following elementary lemma whose proof follows by a standard computation.
	
	\begin{lem}\label{max difference gaus}
		Let $x,y\in\R^N$,  $C_1,C_2>0$ and let us consider the function $$g(\xi)={C_1|x-\xi|^2}-{C_2|y-\xi|^2},\qquad \xi\in\R^N.$$
		If $C_1<C_2$ then $g$ reaches at  $\xi=\frac{C_1 x -C_2y}{C_1-C_2}$ its maximum value $\frac{|x-y|^2}{\frac{1}{C_1}-\frac{1}{C_2}}$.

	\end{lem}

	\begin{teo}\label{Harnack CP}
		There exists $C>0$ such that, for every non-negative solution  $u\in C^{1,2}(]0,T]\times\ov{\R^{N+1}_+})\cap C([0,T]\times\ov{\R^{N+1}_+})$ of the equation $\partial_t u=\mathcal Lu$,  one has
		\begin{align*}
			u(s,z_2)\leq \, C u(t,z_1) \left(\frac t s\right)^{\frac {N+1+c^+} 2}\exp\left({C\frac{| z_1-z_2|^2}{t-s}}\right),\quad  0<s<t \leq T,\quad  z_1,z_2\in\R^{N+1}_+.
		\end{align*}
	\end{teo}
	\begin{proof}
		We employ Theorem \ref{Uniqueness Theorem} to write
		\begin{align*}
			u(t,z)=e^{t\mathcal L}u_0=\int_{\R^{N+1}_+}p_{\mathcal L}\left(t,z,\xi\right)u_0(\xi)\,d\mu(\xi),\qquad t>0, z\in\HS
		\end{align*}
and we remark that, using the semigroup property, we also have  for any $0\leq \tau<t$
		\begin{align*}
			u(t,z)=e^{(t-\tau)\mathcal L}u(\tau,\cdot)=\int_{\R^{N+1}_+}p_{\mathcal L}\left(t-\tau,z,\xi\right)u(\tau,\xi)\,d\mu(\xi).
		\end{align*}
		By Proposition \ref{kernel tilde}, $p_{\mathcal L}$ satisfies 
		\begin{align}\label{up-low Harn-Theo}
			\frac{1}{CV\left(\xi,\sqrt t\right)}\,\exp\left(-C_1\dfrac{|z-\xi|^2}{t}\right)\leq p_{\mathcal L}(t,z,\xi)\leq   \frac{C}{V\left(\xi,\sqrt t\right)}\,\exp\left(-C_2\dfrac{|z-\xi|^2}{t}\right),
		\end{align}
		for some positive constants  $C, C_1, C_2$ satisfying  $C>1$ and $C_1\geq C_2$.

		Let us now fix  $0<s<t$,  $z_1,z_2\in\R^{N+1}_+$. Then  for any $0\leq \tau<s$ one has   
		\begin{align}\label{Harnack-eq1}
			\nonumber u(s,z_2)&=\int_{\R^{N+1}_+}p_{\mathcal L}\left(s-\tau,z_2,\xi\right)u(\tau,\xi)\,d\mu(\xi)\\[1ex]
			&=\int_{\R^{N+1}_+}\frac{p_{\mathcal L}\left(s-\tau,z_2,\xi\right)}{p_{\mathcal L}\left(t-\tau,z_1,\xi\right)}p_{\mathcal L}\left(t-\tau,z_1,\xi\right)u(\tau,\xi)\,d\mu(\xi).
		\end{align}
		Using  \eqref{up-low Harn-Theo} and the doubling condition of Lemma \ref{Misura palle}, the ratio in the above integral satisfies
		\begin{align*}
			\frac{p_{\mathcal L}\left(s-\tau,z_2,\xi\right)}{p_{\mathcal L}\left(t-\tau,z_1,\xi\right)}&\leq  C^2\frac{V(\xi,\sqrt {t-\tau})}{V(\xi,\sqrt {s-\tau})}\exp\left({C_1\frac{|\xi-z_1|^2}{t-\tau}}-{C_2\frac{|\xi-z_1|^2}{s-\tau}}\right) \\[1ex]
			&\leq C^2\left(\frac{t-\tau}{s-\tau}\right)^{\frac{N+1+c^+}{2}}\exp\left({C_1\frac{|\xi-z_1|^2}{t-\tau}}-{C_2\frac{|\xi-z_1|^2}{s-\tau}}\right).
		\end{align*} 
		Let now $0<\epsilon:=\frac{C_2}{C_1}\leq 1$. Requiring that $\frac{C_1}{t-\tau}<\frac{C_2}{s-\tau}$, that is $s-\tau<\epsilon(t-\tau)$, we apply Lemma \ref{max difference gaus} thus obtaining
		\begin{align*}
			\frac{p_{\mathcal L}\left(s-\tau,z_2,\xi\right)}{p_{\mathcal L}\left(t-\tau,z_1,\xi\right)}
			&\leq C^2\left(\frac{t-\tau}{s-\tau}\right)^{\frac{N+1+c^+}{2}}\exp\left({\frac{|z_2-z_1|^2}{\frac{t-\tau}{C_1}-\frac{s-\tau}{C_2}}}\right)\\[1ex]
			&= C^2\left(\frac{t-\tau}{s-\tau}\right)^{\frac{N+1+c^+}{2}}\exp\left(C_2{\frac{|z_2-z_1|^2}{\epsilon(t-\tau)-(s-\tau)}}\right).
		\end{align*}
		Combining the previous inequality with \eqref{Harnack-eq1} we obtain
		\begin{align}\label{Harnack-eq2}
			u(s,z_2)\leq C^2\left(\frac{t-\tau}{s-\tau}\right)^{\frac{N+1+c^+}{2}}\exp\left(C_2{\frac{|z_2-z_1|^2}{\epsilon(t-\tau)-(s-\tau)}}\right)u(t,z_1).
		\end{align}
		Let us now distinguish two cases.
		
		First assume that $0<s\leq\frac{\epsilon}2 t$. Then we use \eqref{Harnack-eq2} with $\tau=0$ to obtain, since $ \epsilon t-s\geq \frac{\epsilon}{2-\epsilon}(t-s)$, 
		\begin{align}\label{Harnack-eq3}
			\nonumber u(s,z_2)&\leq C^2\left(\frac{t}{s}\right)^{\frac{N+1+c^+}{2}}\exp\left(C_2{\frac{|z_2-z_1|^2}{\epsilon t-s}}\right)u(t,z_1)\\[1ex]
				&\leq C^2\left(\frac{t}{s}\right)^{\frac{N+1+c^+}{2}}\exp\left(\frac{C_2(2-\epsilon)}{\epsilon}{\frac{|z_2-z_1|^2}{ t-s}}\right)u(t,z_1).
%\\[1ex]
			%&= C^2\left(\frac{t}{s}\right)^{\frac{N+1+c^+}{2}}\exp\left((2C_1-C_2){\frac{|z_2-z_1|^2}{ t-s}}\right)u(t,z_1).
		\end{align}
		If, instead,  $\frac{\epsilon}2 t< s<t $, we use  \eqref{Harnack-eq2} wit $\tau=s-(t-s)\frac \epsilon 2$. In particular
		\begin{align*}
			\frac \epsilon 2 s<\tau <s
		\end{align*}
		and
		\begin{align*}
			\frac{t-\tau}{s-\tau}=\frac{2+ \epsilon }{ \epsilon }<\frac{2+ \epsilon }{ \epsilon }\frac t s,\qquad \epsilon(t-\tau)-(s-\tau)=\left(\frac \epsilon 2+\frac{\epsilon^2}2\right)(t-s).
		\end{align*} 
		Then \eqref{Harnack-eq2} and the previous inequalities imply that \begin{align}\label{Harnack-eq4}
			u(s,z_2)&\leq  C^2\left(\frac{2+ \epsilon }{ \epsilon }\right)^{\frac{N+1+c^+}{2}}\left(\frac{t}{s}\right)^{\frac{N+1+c^+}{2}}\exp\left(\frac{2C_2}{\epsilon+\epsilon^2}{\frac{|z_2-z_1|^2}{ t-s}}\right)u(t,z_1).
		\end{align}
		%\eqref{Harnack-eq3} and \eqref{Harnack-eq4} prove then the required claim.
	\end{proof}

Note that the constant $C$ in the above theorem is independent of $T$.

\section{Consequences}
The following result  is a direct consequence of  Theorem  \ref{Harnack CP}.
	\begin{cor}\label{Harnack negative time}
	There exists $C>0$ such that, for every non-negative solution  $u\in C^{1,2}(]-\infty,T]\times\ov{\R^{N+1}_+})$ of the equation $\partial_t u=\mathcal Lu$,  one has
	\begin{align*}
		u(s,z_2)\leq \, C u(t,z_1) \exp\left({C\frac{| z_1-z_2|^2}{t-s}}\right),\quad  -\infty<s<t\leq T,\quad  z_1,z_2\in\R^{N+1}_+.
	\end{align*}
\end{cor}
\begin{proof}
For $\tau<T$ fixed, we apply Theorem  \ref{Harnack CP} to $v(t,x)=u(t+\tau,x)$,   $0 \leq t \leq T-\tau$ and we get with $C$ independent of $\tau$
\begin{align*}
	v(s',z_2)\leq \, C v(t',z_1) \left(\frac {t'} {s'}\right)^{\frac {N+1+c^+} 2}\exp\left({C\frac{| z_1-z_2|^2}{t'-s'}}\right),\quad  0<s'<t',\quad  z_1,z_2\in\R^{N+1}_+.
\end{align*}
Setting $s=s'+\tau$, $t=t'+\tau$ and recalling the definition of $v$, the above inequality reads as 
\begin{align*}
	u(s,z_2)\leq \, C u(t,z_1) \left(\frac {t-\tau} {s-\tau}\right)^{\frac {N+1+c^+} 2}\exp\left({C\frac{| z_1-z_2|^2}{t-s}}\right),\quad  \tau<s<t,\quad  z_1,z_2\in\R^{N+1}_+
\end{align*}
and it is sufficient to let $\tau\to-\infty$.
\end{proof}

As for the classical heat equation we deduce a result on the limit at $-\infty$ for global solutions.

\begin{prop}\label{limit negative time}
Let  $u\in C^{1,2}(]-\infty,T]\times\ov{\R^{N+1}_+})$ be a solution of the equation $\partial_t u=\mathcal Lu$.
\begin{itemize}
\item[(i)] If  $L=\inf \{u(t,z),\, t \leq T,\ z\in\ov\HS\}>-\infty$, then 
	$ \ds \lim_{t\to-\infty}u(t,z)=L$ for every  $z \in \HS$. \\
\item[(ii)]
If $M=\sup \{u(t,z),\, t \leq T,\ z\in\ov\HS\}<\infty$,  then 
	$ \ds \lim_{t\to-\infty}u(t,z)=M$ for every  $z \in \HS$.
\end{itemize}
	\end{prop}
\begin{proof}
(i) By replacing $u$ with $u-L \geq 0$ we may assume  that $L=0$. By Corollary \ref{Harnack negative time} we have
	\begin{align*}
	u(s,z)\leq \, C u(t,z_1) \exp\left({C\frac{| z_1-z|^2}{t-s}}\right),\quad  s<t,\quad  z,z_1\in\R^{N+1}_+
\end{align*}
and then	$\ds 0\leq\limsup_{s\to-\infty }u(s,z)\leq \, C u(t,z_1)$.
Taking the $\inf$ with respect to $t$ and $z_1$ we obtain 
$$0\leq\limsup_{s\to-\infty }u(s,z)\leq \, C \inf u=0.$$
Part (ii) follows from (i) by considering $M-u $, whose infimum is 0.
\end{proof}

Liouville-type theorems  easily follow, both for parabolic and elliptic equations.

\begin{cor} \label{Liouville1}
Let  $u\in C^{1,2}(]-\infty,T]\times\ov{\R^{N+1}_+})$  be a bounded solution of $u_t=\mathcal L u$. Then $u$ is constant.
\end{cor}
\begin{proof} We have $\ds \lim_{t \to -\infty}u(t,z)=\inf u=\sup u$. 
\end{proof}

\begin{cor} \label{Liouville}
Let $u \in C^2(\HS)$ be a solution of $\mathcal L u=0$ in $\HS$ bounded from below or from above. Then $u$ is constant.
\end{cor}
\begin{proof} Assume for example that $L=\inf u >-\infty$. Then $u$ is a stationary solution of $u_t=\mathcal L u$ and $\ds u(z)=\lim_{t \to -\infty}u(z)=L$ for every $ z\in\HS$.
\end{proof}

	\bibliography{../TexBibliografiaUnica/References}
	%\bibliography{References}
\end{document}